\documentclass[a4paper,12pt]{amsart}
\title[Galois actions for semifield extensions]{Galois actions for semifield extensions and Galois coverings on tropical curves}
\author{Song JuAe}
\address{Tokyo Metropolitan University, 1-1 Minami-Ohsawa, Hachioji, Tokyo, 192-0397, Japan.}
\email{song-juae@ed.tmu.ac.jp}

\subjclass[2020]{Primary 08A05, 08A35, 15A80; Secondary 14T20}
\keywords{Galois actions for semifield extensions, Galois coverings on tropical curves}

\usepackage{amsmath,amssymb,amscd}
\usepackage{amsthm}
\usepackage{subfigure}
\usepackage[dvips]{graphicx}

\newtheorem{dfn}{Definition}[section]
\newtheorem{thm}[dfn]{Theorem}
\newtheorem{prop}[dfn]{Proposition}
\newtheorem{cor}[dfn]{Corollary}
\newtheorem{rem}[dfn]{Remark}
\newtheorem{ex}[dfn]{Example}

\def\Gamma{\varGamma}

\begin{document}

\maketitle

\begin{abstract}
For a semifield extension $T /S$, an action of a finite group $G$ on $T$ is Galois if $(1)$ the $G$-invariant subsemifield of $T$ is $S$ and $(2)$ subgroups of $G$ whose invariant semifields coincide are equal.
We show that for a finite harmonic morphism between tropical curves $\varphi : \Gamma \to \Gamma^{\prime}$ and an isometric action of a finite group $G$ on $\Gamma$, $\varphi$ is $G$-Galois \cite{JuAe3} if and only if the natural action of $G$ on the rational function semifield $\operatorname{Rat}(\Gamma)$ of $\Gamma$ induced by the action of $G$ on $\Gamma$ is Galois for the semifield extension $\operatorname{Rat}(\Gamma) / \varphi^{\ast}(\operatorname{Rat}(\Gamma^{\prime}))$, where $\varphi^{\ast}(\operatorname{Rat}(\Gamma^{\prime}))$ stands for the pull-back of $\operatorname{Rat}(\Gamma^{\prime})$ by $\varphi$.
\end{abstract}

\section{Introduction}

We call an injective semiring homomorphism between semifields $S \hookrightarrow T$ a \textit{semifield extension}, and write it as $T / S$.
For a semifield extension $T / S$, an action of a finite group $G$ on $T$ is Galois if $(1)$ the $G$-invariant subsemifield $T^G$ of $T$ is $S$ and $(2)$ subgroups of $G$ whose invariant semifields coincide are equal.
For an intermediate semifield $M$ of $T / S$, we write as $G_M$ the subgroup of $G$ such that the restriction of its every element on $M$ is the identity map of $M$.
For semifield extensions, an analogue of Galois correspondence holds:

\begin{thm}[Galois correspondence for semifield extensions]
	\label{thm:Galois}
Let $T / S$ be a semifield extension.
Fix an action of a finite group $G$ on $T$.
Let $A$ be the set of all intermediate semifields $M$ of $T / S$ such that $M = T^{G_M}$.
Let $B$ be the set of all subgroups of $G$.
Then, if the action of $G$ on $T$ is Galois for $T / S$, then the maps $\Phi : A \to B; M \mapsto G_M$ and $\Psi : B \to A; H \mapsto T^H$ satisfy $\Psi \circ \Phi = \operatorname{id}_A$, $\Phi \circ \Psi = \operatorname{id}_B$ and reverse the inclusion relations, where $\operatorname{id}_A$ (resp. $\operatorname{id}_B$) denotes the identity map of $A$ (resp. $B$).
Moreover, for any $M \in A$, the natural action of $G_M$ on $T$ is Galois for $T / M$.
\end{thm}

One of the biggest differences from field extensions is that for the natural action of the automorphism group $\operatorname{Aut}(T / S)$ of a semifield extension $T / S$ on $T$, even the invariant subsemifield of $T$ by $\operatorname{Aut}(T / S)$ is $S$, Theorem \ref{thm:Galois} may not hold.
Here, $\operatorname{Aut}(T / S)$ is the group of all automorphisms of $T$ whose restrictions on $S$ are the identity map of $S$.
It occurs due to that the automorphism group of a semifield extension is more complicated than that of a field extension (see the following example).
Hence we cannot drop the second condition of the definition of Galois actions for semifield extensions.

\begin{ex}
\upshape{
The $\boldsymbol{T}$-algebra automorphism group $\operatorname{Aut}_{\boldsymbol{T}}(\operatorname{Rat}(\Gamma))$ of the rational function semifield $\operatorname{Rat}(\Gamma)$ of a tropical curve $\Gamma$ is isomorphic to the automorphism group $\operatorname{Aut}(\Gamma)$ of $\Gamma$ by \cite[Corollary 1.3]{JuAe4}.
Here, $\boldsymbol{T}$ is the tropical semifield $(\boldsymbol{R} \cup \{ -\infty \}, \operatorname{max}, +)$ and a tropical curve is a metric graph that may have edges of length $\infty$.
$\operatorname{Aut}(\Gamma)$ coincides with the isometry group of $\Gamma$ (except points at infinity).
Hence, Artin's theorem, which states that for a finite group $G$ of automorphisms of a field $L$ and the invariant subfield $K$ of $L$ by $G$, the extension $L / K$ is a finite Galois extension with Galois group $G$, clearly does not hold for semifield extensions.
}
\end{ex}

The following theorem gives a relation between Galois coverings on tropical curves and Galois actions for semifield extensions:

\begin{thm}
	\label{thm:Galoisextension}
Let $\varphi : \Gamma \to \Gamma^{\prime}$ be a finite harmonic morphism between tropical curves and $G$ a finite group isometrically acting on $\Gamma$.
Then, $\varphi$ is $G$-Galois if and only if the action of $G$ on the rational function semifield $\operatorname{Rat}(\Gamma)$ of $\Gamma$ naturally induced by the action of $G$ on $\Gamma$ is Galois for the semifield extension $\operatorname{Rat}(\Gamma) / \varphi^{\ast}(\operatorname{Rat}(\Gamma^{\prime}))$.
\end{thm}

Here, finite harmonic morphisms are morphisms of our category of tropical curves (see Section \ref{Section2} for more details), $\varphi^{\ast}(\operatorname{Rat}(\Gamma^{\prime}))$ stands for the pull-back of $\operatorname{Rat}(\Gamma^{\prime})$ by $\varphi$, and ``$\varphi$ is $G$-Galois'' means that $(1)$ $\varphi$ is a finite harmonic morphism of degree $|G|$ (the order of $G$) and $(2)$ the action of $G$ on $\Gamma$ induces a transitive action on every fiber and $(3)$ every stabilizer subgroup of $G$ with respect to all but a finite number of points is trivial.

This paper is organized as follows.
In Section 2, we prepare basic definitions related to semirings and tropical curves which we need later.
Section 3 gives proofs of Theorems \ref{thm:Galois} and \ref{thm:Galoisextension}.
In that section, we also consider sufficient conditions such that finite group actions become Galois under some assumptions.

\section*{Acknowledgements}
The author thanks her supervisor Masanori Kobayashi, Yuki Kageyama, Daichi Miura, Yasuhito Nakajima, and Ken Sumi for helpful comments.
This work was supported by JSPS KAKENHI Grant Number 20J11910.

\section{Preliminaries}
	\label{Section2}

\subsection{Semirings and congruences}

In this paper, a \textit{semiring} is a commutative semiring with the absorbing neutral element $0$ for addition and the identity $1$ for multiplication such that $0 \not= 1$.
If every nonzero element of a semiring $S$ is multiplicatively invertible, then $S$ is called a \textit{semifield}.
A semiring $S$ is \textit{additively idempotent} if $x + x = x$ for any $x \in S$.
An additively idempotent semiring $S$ has a natural partial order, i.e., for $x, y \in S$, $x \ge y$ if and only if $x + y = x$.

A map $\phi : S_1 \to S_2$ between semirings is a \textit{semiring homomorphism} if for any $x, y \in S_1$,
\begin{align*}
\phi(x + y) = \phi(x) + \phi(y), \	\phi(x \cdot y) = \phi(x) \cdot \phi(y), \	\phi(0) = 0, \	\text{and}\	\phi(1) = 1.
\end{align*}
A semiring homomorphism $\phi : S_1 \to S_2$ is a \textit{semiring isomorphism} if $\phi$ is bijective.
A \textit{semiring automorphism} of $S$ is a semiring isomorphism $S \to S$.

Given a semiring homomorphism $\phi : S_1 \to S_2$, we call the pair $(S_2, \phi)$ (for short, $S_2$) a \textit{$S_1$-algebra}.
For a semiring $S_1$, a map $\psi : (S_2, \phi) \to (S_2^{\prime}, \phi^{\prime})$ between $S_1$-algebras is a \textit{$S_1$-algebra homomorphism} if $\psi$ is a semiring homomorphism and $\phi^{\prime} = \psi \circ \phi$.
When there is no confusion, we write $\psi : S_2 \to S_2^{\prime}$ simply.

The set $\boldsymbol{T} := \boldsymbol{R} \cup \{ -\infty \}$ with two tropical operations:
\begin{align*}
a \oplus b := \operatorname{max}\{ a, b \} \quad	\text{and} \quad a \odot b := a + b,
\end{align*}
where both $a$ and $b$ are in $\boldsymbol{T}$, becomes a semifield.
Here, for any $a \in \boldsymbol{T}$, we handle $-\infty$ as follows:
\begin{align*}
a \oplus (-\infty) = (-\infty) \oplus a = a \quad \text{and} \quad a \odot (-\infty) = (-\infty) \odot a = -\infty.
\end{align*}
$\boldsymbol{T}$ is called the \textit{tropical semifield}.

Let $S$ be a semiring.
A subset $E \subset S \times S$ is a \textit{congruence} on $S$ if it is a subsemiring of $S \times S$ that defines an equivalence relation on $S$.
The \textit{kernel} of a semiring homomorphism $\phi : S_1 \to S_2$ is the congruence $\operatorname{ker} (\phi) = \{ (x, y) \in S_1 \times S_1 \,|\, \phi(x) = \phi(y) \}$.

\subsection{Tropical curves}
In this paper, a \textit{graph} is an unweighted, undirected, finite, connected nonempty multigraph that may have loops.
For a graph $G$, the set of vertices is denoted by $V(G)$ and the set of edges by $E(G)$.
The \textit{degree} of a vertex is the number of edges incident to it.
Here, a loop is counted twice.
A \textit{leaf end} is a vertex of degree one.
A \textit{leaf edge} is an edge incident to a leaf end.

A \textit{tropical curve} is the underlying topological space of the pair $(G, l)$ of a graph $G$ and a \textit{length function} $l: E(G) \to {\boldsymbol{R}}_{>0} \cup \{\infty\}$, where $l$ can take the value $\infty$ on only leaf edges, together with an identification of each edge $e$ of $G$ with the closed interval $[0, l(e)]$.
When $l(e)=\infty$, the interval $[0, \infty]$ is the one point compactification of the interval $[0, \infty)$ and the leaf end of $e$ must be identified with $\infty$.
We regard this not just as a topological space but as almost a metric space.
The distance between $\infty$ and any other point is infinite.
If $E(G) = \{ e \}$ and $l(e)=\infty$, then we can identify either leaf ends of $e$ with $\infty$.
When a tropical curve $\Gamma$ is obtained from $(G, l)$, the pair $(G, l)$ is called a \textit{model} for $\Gamma$.
There are many possible models for $\Gamma$.
A model $(G, l)$ is \textit{loopless} if $G$ is loopless.
Let $\Gamma_{\infty}$ denote the set of all points of $\Gamma$ identified with $\infty$.
An element of $\Gamma_{\infty}$ is called a \textit{point at infinity}.
The \textit{valence} of a point $x$ of $\Gamma$ is the minimum number of the connected components of $U \setminus \{ x \}$ with all neighborhoods $U$ of $x$.
Remark that this ``valence" is defined for a point of a tropical curve and the ``valence" in the first paragraph of this subsection is defined for a vertex of a graph, and these are compatible with each other.
We frequently identify a vertex (resp. an edge) of $G$ with the corresponding point (resp. the corresponding closed subset) of $\Gamma$.
The \textit{relative interior} $e^{\circ}$ of an edge $e$ is $e \setminus \{v, w\}$ with the endpoint(s) $v, w$ of $e$.

\subsection{Rational functions and chip-firing moves}
	\label{semiring:subsection2.3}

Let $\Gamma$ be a tropical curve.
A continuous map $f : \Gamma \to \boldsymbol{R} \cup \{ \pm \infty \}$ is a \textit{rational function} on $\Gamma$ if $f$ is a piecewise affine function with integer slopes, with a finite number of pieces and that can take the value $\pm \infty$ at only points at infinity, or a constant function of $-\infty$.
$\operatorname{Rat}(\Gamma)$ denotes the set of all rational functions on $\Gamma$.
For rational functions $f, g \in \operatorname{Rat}(\Gamma)$ and a point $x \in \Gamma \setminus \Gamma_{\infty}$, we define
\begin{align*}
(f \oplus g) (x) := \operatorname{max}\{f(x), g(x)\} \quad \text{and} \quad (f \odot g) (x) := f(x) + g(x).
\end{align*}
We extend $f \oplus g$ and $f \odot g$ to points at infinity to be continuous on whole $\Gamma$.
Then both are rational functions on $\Gamma$.
Note that for any $f \in \operatorname{Rat}(\Gamma)$, we have
\begin{align*}
f \oplus (-\infty) = (-\infty) \oplus f = f
\end{align*}
and
\begin{align*}
f \odot (-\infty) = (-\infty) \odot f = -\infty.
\end{align*}
Then $\operatorname{Rat}(\Gamma)$ becomes a semifield with these two operations.
Also, $\operatorname{Rat}(\Gamma)$ becomes a $\boldsymbol{T}$-algebra with the natural inclusion $\boldsymbol{T} \hookrightarrow \operatorname{Rat}(\Gamma)$.
Let $\operatorname{Aut}_{\boldsymbol{T}}(\operatorname{Rat}(\Gamma))$) denote the set of all $\boldsymbol{T}$-algebra automorphisms of $\operatorname{Rat}(\Gamma)$.
Then $\operatorname{Aut}_{\boldsymbol{T}}(\operatorname{Rat}(\Gamma))$ has a group structure.
Note that for $f, g \in \operatorname{Rat}(\Gamma)$, $f = g$ means that $f(x) = g(x)$ for any $x \in \Gamma$.

A \textit{subgraph} of a tropical curve is a compact subset of the tropical curve with a finite number of connected components.
Let $\Gamma_1$ be a subgraph of a tropical curve $\Gamma$ which has no connected components consisting of only a point at infinity, and $l$ a positive number or infinity.
The \textit{chip-firing move} by $\Gamma_1$ and $l$ is defined as the rational function $\operatorname{CF}(\Gamma_1, l)(x) := - \operatorname{min}(l, \operatorname{dist}(x, \Gamma_1))$ with $x \in \Gamma$, where $\operatorname{dist}(x, \Gamma_1)$ stands for the distance between $x$ and $\Gamma_1$.

\subsection{Finite harmonic morphisms and Galois coverings}

Let $\varphi : \Gamma \to \Gamma^{\prime}$ be a continuous map between tropical curves.
$\varphi$ is a \textit{finite harmonic morphism} if there exist loopless models $(G, l)$ and $(G^{\prime}, l^{\prime})$ for $\Gamma$ and $\Gamma^{\prime}$, respectively, such that $(1)$ $\varphi(V(G)) \subset V(G^{\prime})$ holds, $(2)$ $\varphi(E(G)) \subset E(G^{\prime})$ holds, $(3)$ for any edge $e$ of $G$, there exists a positive integer $\operatorname{deg}_e(\varphi)$ such that for any points $x, y$ of $e$, $\operatorname{dist}(\varphi (x), \varphi (y)) = \operatorname{deg}_e(\varphi) \cdot \operatorname{dist}(x, y)$ holds, and $(4)$ for every vertex $v$ of $G$, the sum $\sum_{e \in E(G):\, e \mapsto e^{\prime},\, v \in e} \operatorname{deg}_e(\varphi)$ is independent of the choice of $e^{\prime} \in E(G^{\prime})$ incident to $\varphi(v)$.
This sum is denoted by $\operatorname{deg}_v(\varphi)$.
Then, the sum $\Sigma_{v \in V(G):\, v \mapsto v^{\prime}} \operatorname{deg}_v(\varphi)$ is independent of the choice of vertex $v^{\prime}$ of $G^{\prime}$.
It is said the \textit{degree} of $\varphi$.
If both $\Gamma$ and $\Gamma^{\prime}$ are singletons, we regard $\varphi$ as a finite harmonic morphism that can have any number as its degree.
Note that for any $x \in \Gamma$, we can choose loopless models $(G, l), (G^{\prime}, l^{\prime})$ above so that $x \in V(G)$ and $\varphi(x) \in V(G^{\prime})$.

Let $\varphi : \Gamma \to \Gamma^{\prime}$ be a finite harmonic morphism between tropical curves.
For $f$ in $\operatorname{Rat}(\Gamma)$, the push-forward of $f$ is the function $\varphi_{\ast}(f) : \Gamma^{\prime} \to \operatorname{R} \cup \{ \pm \infty \}$ defined as follows:
for $x^{\prime} \in \Gamma^{\prime} \setminus \Gamma^{\prime}_{\infty}$, 
\begin{align*}
\varphi_{\ast}(f)(x^{\prime}) := \sum_{\substack{x \in \Gamma :\, \varphi(x) = x^{\prime}}} \operatorname{deg}_x(\varphi) \cdot f(x).
\end{align*}
We continuously extend $\varphi_{\ast}(f)$ on $\Gamma^{\prime}_{\infty}$.
Then, $\varphi_{\ast}(f)$ is a rational function on $\Gamma^{\prime}$.
The pull-back $\varphi^{\ast}(f^{\prime})$ of $f^{\prime} \in \operatorname{Rat}(\Gamma^{\prime})$ is the rational function $f^{\prime} \circ \varphi$ on $\Gamma$.

\begin{rem}
	\label{rem:isom}
\upshape{
Let $\varphi : \Gamma \to \Gamma^{\prime}$ be a map between tropical curves.
Then $\varphi$ is a continuous map whose restriction on $\Gamma \setminus \Gamma_{\infty}$ is an isometry if and only if it is a finite harmonic morphism of degree one.
In this paper, we will use the word ``a finite group $G$ isometrically acts on a tropical curve $\Gamma$" as the meaning that $G$ continuously acts on $\Gamma$ and it is isometric on $\Gamma \setminus \Gamma_{\infty}$.
}
\end{rem}

\begin{rem}
	\label{rem:pi1}
\upshape{
Let $\Gamma$ be a tropical curve and $G$ a finite group isometrically acting on $\Gamma$.
Let $\Gamma / G$ be the quotient space (as topological space) and $\pi_G : \Gamma \to \Gamma / G$ be the natural surjection.
Fix a loopless model $(V, E, l)$ ($V$ is a set of vertices and $E$ is a set of edges) for $\Gamma$ compatible with the action of $G$ on $\Gamma$, i.e., for any $g \in G$, $g(V) = V$ holds.
Let $V^{\prime} := \pi_G(V)$, $E^{\prime} := \pi_G(E)$ and for any $e \in E$, $l^{\prime}(\pi_G(e)) := |G_e| \cdot l(e)$, where $G_e$ denotes the stabilizer subgroup of $G$ with respect to $e$.
Then, $(V^{\prime}, E^{\prime}, l^{\prime})$ gives $\Gamma / G$ a tropical curve structure and is a loopless model for the quotient tropical curve $\Gamma / G$.
By loopless models $(V, E, l)$ and $(V^{\prime}, E^{\prime}, l^{\prime})$ for $\Gamma$ and $\Gamma / G$, respectively, $\pi_G$ is a finite harmonic morphism of degree $|G|$.
}
\end{rem}

\begin{dfn}[{\cite[Definition 4.1]{JuAe3}}]
	\label{dfn:Galois1}
\upshape{
Let $\Gamma$ be a tropical curve and $G$ a finite group.
An isometric action of $G$ on $\Gamma$ is \textit{Galois} if there exists a finite subset $U^{\prime}$ of $\Gamma / G$ such that for any $x^{\prime} \in (\Gamma / G) \setminus U^{\prime}$, $|\pi^{-1}_G(x^{\prime})| = |G|$ holds.
}
\end{dfn}

\begin{dfn}[{\cite[Definition 4.2]{JuAe3}}]
	\label{dfn:Galois2}
\upshape{
Let $\varphi : \Gamma \to \Gamma^{\prime}$ be a map between tropical curves.
$\varphi$ is \textit{Galois} if there exists a Galois action of a finite group $G$ on $\Gamma$ such that there exists a finite haramonic morphism of degree one $\theta : \Gamma / G \to \Gamma^{\prime}$ satisfying $\varphi \circ g = \theta \circ \pi_G$ for any $g \in G$.
Then, we say that $\varphi$ is a \textit{$G$-Galois covering} on $\Gamma^{\prime}$ or just \textit{$G$-Galois}.
}
\end{dfn}

\section{Main results}

Throughout this paper, we assume that an action of a finite group $G$ on a semifield $T$ induces a group homomorphism from $G$ to the automorphism group of $T$.

\begin{dfn}[Semifield extensions]
\upshape{
Let $T, S$ be semifields.
We call an injective semiring homomorphism $S \hookrightarrow T$ a \textit{semifield extension}, and write it as $T / S$.
We frequently identify a semifield extension $T / S$ with the inclusion $S \subset T$ via the injection $S \hookrightarrow T$.

Let $T / S$ be a semifield extension.
Let $M$ be a semifield.
A pair of injective semiring homomorphisms $S \hookrightarrow M$, $M \hookrightarrow T$ compatible with $T / S$ is called an \textit{intermediate semifield} of $T / S$.
We frequently identify the intermediate semifield with the inclusion $S \subset M \subset T$ via the injections above. 
We also call $M$ an \textit{intermediate semifield} of $T / S$.

We call an automorphism of $T$ whose restriction on $S$ is the identity map of $S$ an \textit{automorphism} of $T / S$.
Let $\operatorname{Aut}(T / S)$ denote the set of all automorphisms of $T / S$.
Then, $\operatorname{Aut}(T / S)$ becomes a group.
We call it the \textit{automorphism group} of $T / S$.

Let $G$ be a finite group.
For an action of $G$ on $T$, we call the subset of $T$ whose each element is fixed by all elements of $G$ the \textit{$G$-invariant semifield}, and write it as $T^G$.
Then, $T^G$ naturally becomes an intermediate semifield of $T / S$.
The action of $G$ on $T$ is \textit{Galois} for $T / S$ if $(1)$ $T^G = S$ and $(2)$ subgroups of $G$ whose invariant semifields coincide are equal.
For an intermediate semifield $M$ of $T / S$, we write as $G_M$ the subset of $G$ such that the restriction of its every element on $M$ is the identity map of $M$.
Then, $G_M$ becomes a group.
}
\end{dfn}

\begin{rem}
\upshape{
Note that there exists a non-injective semiring homomorphism between semifields.
In fact, the map $\boldsymbol{T} \to \boldsymbol{B}; t \not= -\infty \mapsto 0; -\infty \mapsto -\infty$ is a non-injective semiring homomorphism.
Here, $\boldsymbol{B}$ is the \textit{boolean algebra} $(\{ -\infty, 0 \}, \operatorname{max}, +)$, which is a subsemifield of $\boldsymbol{T}$.
}
\end{rem}

\begin{prop}
	\label{prop:Galois}
Let $T / S$ be a semifield extension.
Fix an action of a finite group $G$ on $T$.
Let $A$ be the set of all intermediate semifields $M$ of $T / S$ such that $M = T^{G_M}$.
Let $B$ be the set of $G_{T^H}$ for any subgroup $H$ of $G$.
Then, the maps $\Phi : A \to B; M \mapsto G_M$ and $\Psi : B \to A; G_{T^H} \mapsto T^{G_{T^H}}$ satisfy $\Psi \circ \Phi = \operatorname{id}_A$, $\Phi \circ \Psi = \operatorname{id}_B$ and reverse the inclusion relations. 
\end{prop}

\begin{proof}
It is straightforward.
\end{proof}

\begin{proof}[Proof of Theorem \ref{thm:Galois}]
Let $H \in B$.
Since $T^H = T^{G_{T^H}}$, $\Psi(H) \in A$.
We shall show the last assertion.
Let $M \in A$.
Since $M \in A$, $M = T^{G_M}$ holds.
For any subgroups $H_1, H_2$ of $G_M$, if $T^{H_1} = T^{H_2}$, then $H_1 = H_2$ holds since $H_1$, $H_2$ are subgroups of $G$.
Hence, the natural action of $G_M$ on $T$ is Galois for $T / M$.
The remaining assertions are shown by Proposition \ref{prop:Galois}.
\end{proof}

\begin{rem}
\upshape{
Let $T / S$ be a semifield extension.
Let $G$ be a finite group acting on $T$.
When $T^G = S$, the following are equivalent:

$(1)$ the action of $G$ on $T$ is Galois for $T / S$,

$(2)$ for any subgroups $H_1, H_2$ of $G$, if $T^{H_1} = T^{H_2}$, then $H_1 = H_2$,

$(3)$ for any subgroups $H_1, H_2$ of $G$, if $G_{T^{H_1}} = G_{T^{H_2}}$, then $H_1 = H_2$,

$(4)$ for any subgroups $H_1, H_2$ of $G$, if $\operatorname{Aut}(T / T^{H_1}) = \operatorname{Aut}(T / T^{H_2})$, then $H_1 = H_2$, and

$(5)$ for $T / S$, the Galois correspondence holds.
}
\end{rem}

\begin{proof}
$(1) \Leftrightarrow (2) :$ clear.

$(2) \Rightarrow (3) :$ we shall show the contraposition.
Assume that there exist distinct subgroups $H_1, H_2$ of $G$ such that $G_{T^{H_1}} = G_{T^{H_2}}$.
For any $a \in T^{H_1}$ and $f \in G_{T^{H_2}}$, $f(a) = a$ holds since $G_{T^{H_1}} = G_{T^{H_2}}$.
Since $H_2$ is a subgroup of $G_{T^{H_2}}$, we have $a \in T^{H_2}$.
The converse inclusion is shown in a similar way, we have the conclusion.

$(3) \Rightarrow (4) :$ since the intersection of $G$ and $\operatorname{Aut}(T / T^{H_i})$ is $G_{T^{H_i}}$, by $(3)$, we have $H_1 = H_2$.

$(4) \Rightarrow (2) :$ for any subgroups $H_1, H_2$ of $G$, if $T^{H_1} = T^{H_2}$, then since $\operatorname{Aut}(T / T^{H_1}) = \operatorname{Aut}(T / T^{H_2})$, we have $H_1 = H_2$ by $(4)$.

$(1) \Rightarrow (5) :$ it is given by Theorem \ref{thm:Galois}.

$(5) \Rightarrow (2) :$ for any subgroups $H_1, H_2$ of $G$, assume that $T^{H_1} = T^{H_2}$ holds.
Since $G_{T^{H_1}} = G_{T^{H_2}}$, by Theorem \ref{thm:Galois}, we have $H_1 = H_2$.
\end{proof}

\begin{thm}
Let $T / S$ be a semifield extension.
Let $G$ be a finite group acting on $T$.
Let $M$ be an intermediate semifield of $T / S$ such that $M = T^{G_M}$.
Then, the following are equivalent:

$(1)$ $G_M$ is a normal subgroup of $G$,

$(2)$ for any $g \in G$, $g(M) \subset M$,

$(3)$ for any $g \in G$, $g(M) = M$,

$(4)$ there exists a subgroup $H$ of $\operatorname{Aut}(M / S)$ such that for any $g \in G$, there exists $h \in H$ that coincides with the restriction $g|_M$, and

$(5)$ there exists a subgroup $H^{\prime}$ of $\operatorname{Aut}(M / S)$ such that for any $m \in M$, the orbits $Gm$ and $H^{\prime}m$ coincide.

Moreover, if the action of $G$ on $T$ is Galois for $T / S$, then $M^{H^{\prime}} = S$ holds and $H$ can be choosen to be the natural action of $H$ on $T$ is Galois for $T / M$, and if the natural action of $H$ on $T$ is Galois for $T / M$, then $H$ is isomorphic to the quotient group $G / G_M$.
\end{thm}

\begin{proof}
$(1) \Rightarrow (2) :$ for any $g \in G$, $g^{\prime} \in G_M$, $m \in M$, as $g^{-1}g^{\prime}g \in G_M$, we have $g^{-1}g^{\prime}g(m) = m$.
Thus, $g^{\prime}(g(m)) = g(m) \in T^{G_M} = M$ holds.

$(2) \Rightarrow (3) :$ since $g$ is arbitrary element of $G$, we have also $g^{-1}(M) \subset M$.
Hence, $M = g(g^{-1}(M)) \subset g(M)$ holds.

$(3) \Rightarrow (4) :$ we can define a group homomorphism $\phi : G \to \operatorname{Aut}(M / S)$ as $g \mapsto g|_M$.
The image $\operatorname{Im}(\phi)$ is the desired group.

$(4) \Rightarrow (5) :$ it is enough to choose the subgroup $\{ g|_M \,|\, g \in G \}$ of $H$ as $H^{\prime}$.

$(5) \Rightarrow (2) :$ it is clear.

$(2) \Rightarrow (1) :$ for any $g \in G$, $g^{\prime} \in G_M$, $m \in M$, since $g(m) \in M$, we have $g^{-1}g^{\prime}g(m) = g^{-1}(g^{\prime}(g(m))) = g^{-1}(g(m)) = m$.
Therefore, we have $g^{-1}g^{\prime}g \in G_M$.

Assume that the action of $G$ on $T$ is Galois for $T / S$.
It is clear that $M^{H^{\prime}} \supset S$ holds.
For any $a \in M^{H^{\prime}}$, since $\{ a \} = H^{\prime}a = Ga$, by the orbit-stabilizer theorem (cf. \cite[Chapter 6]{Clive=Stuart}), we have $G_a = G$.
Thus, $a \in T^G = S$ holds.
This means that $M^{H^{\prime}} = S$.

Assume that the action of $G$ on $T$ is Galois for $T / S$.
Let $H = \{ g|_M \,|\, g \in G \}$.
Note that it is $\operatorname{Im}(\phi)$.
Let $H_1, H_2$ be subgroups of $H$ satisfying $M^{H_1} = M^{H_2}$.
Let $G_i := \phi^{-1}(H_i)$.
Since
\begin{align*}
M^{H_i} &= \{ m \in M \,|\, \text{for any }h_i \in H_i, h_i(m) = m \}\\
&= \{ m \in M \,|\, \text{for any }g_i \in G_i, g_i(m) = m \}\\
&= \{ t \in T \,|\, \text{for any }g_i \in G_i, g_i(t) = t \} \cap M\\
&= T^{G_i} \cap M,
\end{align*}
we have $T^{G_1} \cap M = T^{G_2} \cap M$.
Since the kernel $\operatorname{Ker}(\phi)$ of $\phi$ and $G_M$ coincide, $G_M$ is a subgroup of $G_i$.
Hence, $M = T^{G_M} \supset T^{G_i}$ hold.
Therefore, we have $T^{G_i} \cap M = T^{G_i}$, and thus, $T^{G_1} = T^{G_2}$.
Since the action of $G$ on $T$ is Galois for $T / S$, $G_1$ must be $G_2$.
Thus, $H_1 = \phi(G_1) = \phi(G_2) = H_2$ hold, and hence, the action of $H$ on $M$ is Galois for $M / S$.

Assume that the action of $H$ on $M$ is Galois for $M / S$.
By assumption, $H^{\prime} := \{ g|_M \,|\, g \in G \}$ is a subgroup of $H$.
By the same argument above, since $T^G = S = M^{H^{\prime}} \supset M^H \supset S$, we have $M^{H^{\prime}} = M^H$.
Since the action of $H$ on $M$ is Galois for $M / S$, we have $H^{\prime} = H$.
Since $H^{\prime} = \operatorname{Im}(\phi)$ and $\operatorname{Ker}(\phi) = G_M$, $H^{\prime}$ is isomorphic to $G / G_M$.
In conclusion, $H$ is isomorphic to $G / G_M$.
\end{proof}

We give a sufficient condition that for a semifield extension $T / S$, a given finite group action on $T$ is Galois:

\begin{prop}
	\label{prop:sufficient1}
Let $T / S$ be a semifield extension.
Fix an action of a finite group $G$ on $T$.
If $T^G = S$ holds and if for any subgroup $H$ of $G$, there exists an element $a \in T$ whose stabilizer subgroup $G_a$ is $H$, then the action of $G$ on $T$ is Galois for $T / S$.
\end{prop}

\begin{proof}
Let $H_1, H_2$ be subgroups of $G$.
Assume that $T^{H_1} = T^{H_2}$ holds.
There exist $a, b \in T$ whose stabilizer subgroups are $H_1, H_2$, respectively.
Hence, we have $H_1 = \bigcap_{c \in T^{H_1}} G_c = \bigcap_{d \in T^{H_2}} G_d = H_2$.
\end{proof}

\begin{prop}
Let $\varphi_i : \Gamma \to \Gamma_i^{\prime}$ be a finite harmonic morphism between tropical curves.
Then, the pull-backs $\varphi^{\ast}_1(\operatorname{Rat}(\Gamma_1^{\prime}))$ and $\varphi^{\ast}_2(\operatorname{Rat}(\Gamma_2^{\prime}))$ coincide if and only if there exists a finite harmonic morphism of degree one $\varphi_{12}$ satisfying $\varphi_1 = \varphi_{12} \circ \varphi_2$.
\end{prop}

\begin{proof}
We show the if part.
Let $f \in \varphi_1^{\ast}(\operatorname{Rat}(\Gamma_1^{\prime}))$.
There exists $f^{\prime} \in \operatorname{Rat}(\Gamma_1^{\prime})$ such that $f = \varphi_1^{\ast}(f^{\prime})$.
Hence we have $f = \varphi_1^{\ast}(f^{\prime}) = f^{\prime} \circ \varphi_1 = f^{\prime} \circ \varphi_{12} \circ \varphi_2$.
Since $\varphi_{12}$ is a finite harmonic morphism, we have $f^{\prime} \circ \varphi_{12} \in \operatorname{Rat}(\Gamma_2^{\prime})$, and thus $f \in \varphi_2^{\ast}(\operatorname{Rat}(\Gamma_2^{\prime}))$.
Since $\varphi_{12}$ is a finite harmonic morphism of degree one, it is bijective and the inverse map $\varphi_{12}^{-1}$ is also a finite harmonic morphism of degree one.
Therefore, we have the inverse inclusion by the same argument.

We show the only if part.
Let $(V, E, l), (V^{\prime}_1, E^{\prime}_1, l^{\prime}_1)$ be loopless models for $\Gamma, \Gamma_1^{\prime}$, respectively, such that $\varphi_1(V) = V^{\prime}_1$.
Let $(\widetilde{V}, \widetilde{E}, \widetilde{l}), (V^{\prime}_2, E^{\prime}_2, l^{\prime}_2)$ be loopless models for $\Gamma, \Gamma_2^{\prime}$, respectively, such that $\varphi_2(\widetilde{V}) = V^{\prime}_2$.
For any $x \in \Gamma \setminus (V \cup \widetilde{V})$, there exist $e \in E$ and $\widetilde{e} \in \widetilde{E}$ containing $x$.
Then, we have $\operatorname{deg}_x(\varphi_1) = \operatorname{deg}_e(\varphi_1) = \operatorname{deg}_{\widetilde{e}}(\varphi_2) = \operatorname{deg}_x(\varphi_2)$ by the definition of pull-back of rational functions.
In fact, there exists a positive number $\varepsilon$ such that $\varphi_1^{\ast}(\operatorname{CF}(\{ \varphi_1(x) \}, \varepsilon))$ (resp. $\varphi_2^{\ast}(\operatorname{CF}(\{ \varphi_2(x) \}, \varepsilon))$) has slope $\operatorname{deg}_e(\varphi_1)$ (resp. $\operatorname{deg}_{\widetilde{e}}(\varphi_2)$) on $\varphi_1^{-1}(U_1^{\prime}) \cap e$ (resp. $\varphi_2^{-1}(U_2^{\prime} \cap \widetilde{e}$)), where $U_i^{\prime}$ is the $\varepsilon$-neighborhood of $\varphi_i(x)$.
Since these slopes $\operatorname{deg}_e(\varphi_1)$ and $\operatorname{deg}_{\widetilde{e}}(\varphi_2)$ are the minimum (absolute values of) slopes other than zero on $e$ and $\widetilde{e}$, respectively, and $\varphi_1^{\ast}(\operatorname{Rat}(\Gamma_1^{\prime})) = \varphi_2^{\ast}(\operatorname{Rat}(\Gamma_2^{\prime}))$, $\operatorname{deg}_e(\varphi_1)$ must be $\operatorname{deg}_{\widetilde{e}}(\varphi_2)$ (cf. \cite[Remark 3.3.24]{JuAe2}).
Also, by the definition of the push-forward of rational functions, with a sufficiently small positive number $\delta$, we have $(\varphi_1)_{\ast}(\operatorname{CF}(\{ x \}, \delta) \odot \delta) = \operatorname{CF}(\{ \varphi_1(x) \}, \operatorname{deg}_e(\varphi_1) \cdot \delta) \odot (\operatorname{deg}_e(\varphi_1) \cdot \delta)$ and $(\varphi_2)_{\ast}(\operatorname{CF}(\{ x \}, \delta) \odot \delta) = \operatorname{CF}(\{ \varphi_2(x) \}, \operatorname{deg}_{\widetilde{e}}(\varphi_2) \cdot \delta) \odot (\operatorname{deg}_{\widetilde{e}}(\varphi_2) \cdot \delta)$.
Since $\varphi_i$ is continuous, the map $\varphi_{12} : \Gamma_2^{\prime} \to \Gamma_1^{\prime}; \varphi_2(y) \to \varphi_1(y)$ is a finite haromonic morphism of degree one, where $y \in \Gamma$.
\end{proof}

\begin{cor}
	\label{cor:pull-back}
Let $\varphi : \Gamma \to \Gamma^{\prime}$ be a finite harmonic morphism between tropical curves.
Let $G$ be a finite group isometrically acting on $\Gamma$.
Then, $\operatorname{Rat}(\Gamma)^G = \varphi^{\ast}(\operatorname{Rat}(\Gamma^{\prime}))$ if and only if $\varphi$ is $G$-Galois.
\end{cor}

\begin{proof}
Let $(V, E, l), (V^{\prime}, E^{\prime}, l^{\prime})$ be loopless models for $\Gamma$, $\Gamma^{\prime}$, respectively, such that $\varphi(V) = V^{\prime}$.
The if part follows from \cite[Remark 3.3.24]{JuAe2} since for any $e \in E$, $\operatorname{deg}_e(\varphi) = 1$.
We shall show the only if part.
If $\operatorname{Rat}(\Gamma)^G = \varphi^{\ast}(\operatorname{Rat}(\Gamma^{\prime}))$, then for any $e \in E$, $\operatorname{deg}_e(\varphi) = 1$.
For any $x \in \Gamma \setminus \Gamma_{\infty}$ and any $\varepsilon > 0$, since $\varphi^{\ast}(\operatorname{CF}(\{ \varphi(x) \}, \varepsilon))$ takes zero at and only at each element of $\varphi^{-1}(\varphi(x))$ and is $G$-invariant, we have $Gx = \varphi^{-1}(\varphi(x))$.
Thus, $\varphi$ is $G$-Galois.
\end{proof}

\begin{proof}[Proof of Theorem \ref{thm:Galoisextension}]
The if part follows from Corollary \ref{cor:pull-back}.
We shall show the only if part.
By Corollary \ref{cor:pull-back}, $\operatorname{Rat}(\Gamma)^G = \varphi^{\ast}(\operatorname{Rat}(\Gamma^{\prime}))$.
Let $G_1, G_2$ be subgroups of $G$.
Assume that $\operatorname{Rat}(\Gamma)^{G_1} = \operatorname{Rat}(\Gamma)^{G_2}$ holds.
Let $\Gamma_i^{\prime}$ be the quotient tropical curve of $\Gamma$ by $G_i$.
By \cite[Theorem 1.1]{JuAe3}, the natural surjection $\pi_{G_i} : \Gamma \to \Gamma_i^{\prime}$ is $G_i$-Galois. Hence, there eixsts a finite harmonic morphism of degree one $\pi_{12}$ satisfying $\pi_{G_1} = \pi_{12} \circ \pi_{G_2}$.
Thus, by \cite[Theorem 1.1]{JuAe3} again, we have $G_1 = G_2$, which completes the proof.
\end{proof}

Note that by \cite[Corollary 1.3]{JuAe4}, the automorphism group of a tropical curve $\Gamma$ is isomorphic to $\operatorname{Aut}_{\boldsymbol{T}}(\operatorname{Rat}(\Gamma))$.
Here, an \textit{automorphism} of $\Gamma$ is a finite harmonic morphism of degree one $\Gamma \to \Gamma$.
Hence, the following corollary holds:

\begin{cor}
Let $\varphi : \Gamma \to \Gamma^{\prime}$ be a finite harmonic morphism between tropical curves and $G$ a finite group isometrically acting on $\operatorname{Rat}(\Gamma)$.
Then, $\varphi$ is $G$-Galois for the natural action of $G$ on $\Gamma$ if and only if the action of $G$ on $\operatorname{Rat}(\Gamma)$ is Galois for $\operatorname{Rat}(\Gamma) / \varphi^{\ast}(\operatorname{Rat}(\Gamma^{\prime}))$.
\end{cor}

\begin{prop}
Let $\varphi : \Gamma \to \Gamma^{\prime}$ be a finite harmonic morphism between tropical curves.
Let $G$ be a finite group isometrically acting on $\Gamma$.
If $\varphi$ is $G$-Galois, then for any subgroup $H$ of $G$, there exists a rational function $f$ on $\Gamma$ whose stabilizer subgroup is $H$.
\end{prop}

\begin{proof}
If $\Gamma$ (and hence $\Gamma^{\prime}$) is a singleton, then the assertion is clear.
Assume that $\Gamma$ (and thus $\Gamma^{\prime}$) is not a singleton.
Let $\Gamma^{\prime \prime}$ be the quotient tropical curve of $\Gamma$ by $H$.
Let $x^{\prime \prime} \in \Gamma^{\prime \prime}$ be a two valent point.
There exists a positive real number $\varepsilon$ such that the $\varepsilon$-neighborhood $U^{\prime \prime}$ of $x^{\prime \prime}$ consists of only two valent points.
Since $\varphi$ is $G$-Galois, the stabilizer subgroup of $G$ with respect to the pull-back $\pi_H^{\ast}(\operatorname{CF}(\{ x^{\prime \prime} \}, \varepsilon))$ is $H$.
\end{proof}

By this proposition and Proposition \ref{prop:sufficient1}, we have the following corollary:

\begin{cor}
Let $\varphi : \Gamma \to \Gamma^{\prime}$ be a finite harmonic morphism between tropical curves.
Let $G$ be a finite group isometrically acting on $\Gamma$.
Then, $\operatorname{Rat}(\Gamma)^G = \varphi^{\ast}(\operatorname{Rat}(\Gamma^{\prime}))$ and for any subgroup $H$ of $G$, there exists $g \in \operatorname{Rat}(\Gamma)$ whose stabilizer subgroup is $H$ if and only if the natural action of $G$ on $\operatorname{Rat}(\Gamma)$ is Galois for $\operatorname{Rat}(\Gamma) / \varphi^{\ast}(\operatorname{Rat}(\Gamma^{\prime}))$.
\end{cor}

\begin{figure}[h]
\begin{center}
\includegraphics[width=0.9\columnwidth]{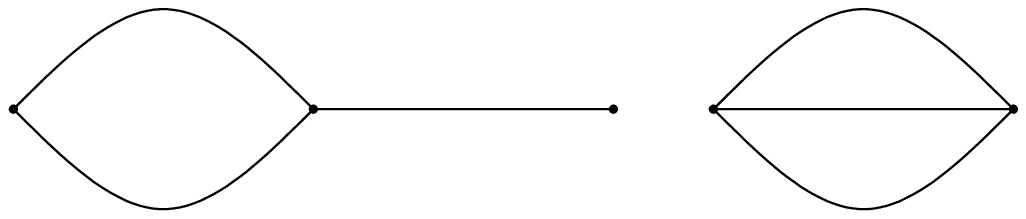}

\vspace{-1mm}
\caption{On each figure, black dots (resp. lines) stand for vertices (resp. edges).}
	\label{figure1}
\end{center}
\end{figure}

\begin{rem}
\upshape{
Let $\varphi : \Gamma \to \Gamma^{\prime}$ be a finite harmonic morphism between tropical curves.
Let $G$ be a finite group isometrically acting on $\Gamma$.
Then, even if the natural action of $G$ on $\operatorname{Rat}(\Gamma)$ is Galois for $\operatorname{Rat}(\Gamma) / \operatorname{Rat}(\Gamma)^G$, it may not be Galois for $\operatorname{Rat}(\Gamma) / \varphi^{\ast}(\operatorname{Rat}(\Gamma^{\prime}))$.
See Example \ref{ex:1}.
}
\end{rem}

\begin{ex}
	\label{ex:1}
\upshape{
Let $G$ be the graph consisting of three vertices $v_1, v_2, v_3$, two multiple edges $e_1, e_2$ between $v_1$ and $v_2$, and one edge $e_3$ between $v_2$ and $v_3$ (the left figure of Figure \ref{figure1}).
Let $\Gamma$ be the tropical curve obtained from $(G, l)$, where $l(E(G)) = \{ 1 \}$.
Let $\sigma$ be the permutation $(e_1 e_2)$.
The group $\langle \sigma \rangle$ generated by $\sigma$ naturally acts on $\Gamma$.
Let $\Gamma^{\prime}$ be the quotient tropical curve of $\Gamma$ by $\langle \sigma \rangle$.
Note that the pair of the quotient graph $G^{\prime} := G / \langle \sigma \rangle$ and the length function $l^{\prime} : E(G^{\prime}) \to \boldsymbol{R} \cup \{ \infty \}; [e_1] \mapsto 1; [e_3] \mapsto 2$ is a model for $\Gamma^{\prime}$, where $[e_i]$ denotes the equivalence class of $e_i$.
The natural surjection $\pi_{\langle \sigma \rangle} : \Gamma \to \Gamma^{\prime}$ is not $\langle \sigma \rangle$-Galois.
Also, the natural action of $\langle \sigma \rangle$ on $\operatorname{Rat}(\Gamma)$ is not Galois for $\operatorname{Rat}(\Gamma) / \pi_{\langle \sigma \rangle}^{\ast}(\operatorname{Rat}(\Gamma^{\prime}))$ but $\operatorname{Rat}(\Gamma) / \operatorname{Rat}(\Gamma)^{\langle \sigma \rangle}$.
}
\end{ex}

\begin{rem}
\upshape{
Let $\Gamma$ be a tropical curve.
Then, even an isometric action of a finite group $G$ on $\Gamma$ is faithful, the natural action of $G$ on $\operatorname{Rat}(\Gamma)$ may not be Galois for $\operatorname{Rat}(\Gamma) / \operatorname{Rat}(\Gamma)^G$.
See Example \ref{ex:2}.
}
\end{rem}

\begin{ex}
	\label{ex:2}
\upshape{
Let $G$ be the graph consisting of two vertices and three multiple edges $e_1, e_2, e_3$ between them (the right figure of Figure \ref{figure1}).
Let $l$ be the length function such that $l(E(G)) = \{ 1 \}$ and $\Gamma$ the tropical curve obtained from $(G, l)$.
The symmetric group of degree three $\Sigma_3$ isometrically and faithfully acts on $\Gamma$ in a natural way.
The invariant subsemifields by the permutation $(e_1 e_2 e_3)$ and by $\Sigma_3$ coincide.
On the other hand, clearly $\langle (e_1 e_2 e_3) \rangle \not= \Sigma_3$.
Hence the natural action of $\Sigma_3$ on $\operatorname{Rat}(\Gamma)$ is not Galois for $\operatorname{Rat}(\Gamma) / \operatorname{Rat}(\Gamma)^{\Sigma_3}$.
}
\end{ex}

\begin{prop}
	\label{prop:totally ordered}
Let $S$ be an additively idempotent semiring and $\phi : S \to S$ an automorphism of $S$.
If $S$ is totally ordered with respect to the natural partial order, then $\phi$ is the identity of $S$ or the order of $\phi$ is infinite.
\end{prop}

\begin{proof}
Assume that $\phi$ is not the identity of $S$.
Then, there exists an element $s \in S \setminus \{0, 1\}$ such that $\phi(s) \not= s$.
Since $S$ is totally ordered, $\phi(s) + s$ is $s$ or $\phi(s)$.
If $\phi(s) + s = s$, then $\phi^2(s) + \phi(s) = \phi(s)$.
By repeating the same argument, we have $s \ge \phi(s) \ge \phi^2(s) \ge \cdots$.
Since $\phi(s) \not= s$, these are distinct.
Hence, in this case, the cardinality of $\langle \phi \rangle s$ is infinite.
When $\phi(s) + s = \phi(s)$, by the same argument, the cardinality of $\langle \phi \rangle s$ is also infinite.
\end{proof}

By Proposition \ref{prop:totally ordered}, for a semifield extension $T / S$, if $T$ is an additively idempotent semiring totally ordered with respect to the natural partial order and an action of a finite group $G$ on $T$ is Galois for $T / S$, then $G$ is trivial and $T = S$.

Finally, we consider another sufficient condition that for a semifield extension $T / S$, a given finite group action on $T$ is Galois under some assumptions.
Let $U$ be an additively idempotent semifield.
Assume that $U$ is totally ordered with respect to the natural partial order.
Let $T$ be a finitely generated semifield over $U$.
Let $t_1, \ldots, t_n$ be generators of $T$.
Then, the $U$-algebra homomorphism $\psi : U(X_1, \ldots, X_n) \to T$ defined by $X_i \mapsto t_i$ is surjective, where $U(X_1, \ldots, X_n)$ denotes the semifield of all fractions of all polynomials with coefficients in $U$ and each $X_i$ is an indeterminate.
By \cite[Proposition 2.4.4]{Giansiracusa=Giansiracusa}, $\psi$ induces a semiring isomorphism $\phi : U(X_1, \ldots, X_n) / \operatorname{ker} (\psi) \to T$.
Let $V$ be the set $\{ u \in (U \setminus \{ 0 \})^n \,|\, \text{for any }(f, f^{\prime}) \in \operatorname{ker}(\psi), f(u) = f^{\prime}(u) \}$.
Assume that for any two elements $[f], [f^{\prime}] \in U(X_1, \ldots, X_n) / \operatorname{ker}(\psi)$, $[f] = [f^{\prime}]$ if and only if for any $v \in V$, $f(v) = f^{\prime}(v)$ holds.
Let $G$ be a finite group.
Assume that an action of $G$ on $V$ induces the action of $G$ on $U(X_1, \ldots, X_n) / \operatorname{ker}(\psi)$ such that for any $g \in G$, $[f] \in U(X_1, \ldots, X_n) / \operatorname{ker}(\psi)$, $v \in V$, $g([f])(v) = f(g^{-1}(v))$.
Since $U(X_1, \ldots, X_n) / \operatorname{ker}(\psi)$ is a $U$-algebra, this induced action indeces a group homomorphism from $G$ to the $U$-algebra automorphism group of $U(X_1, \ldots, X_n) / \operatorname{ker}(\psi)$.

\begin{prop}
In the above setting, if there exists an element $v$ of $V$ whose stabilizer subgroup of $G$ is trivial, then the action of $G$ on $T$ is Galois for $T / T^G$.
\end{prop}

\begin{proof}
Let $v = (v_1, \ldots, v_n)$ and 
\begin{align*}
f(X_1, \ldots, X_n) := \left[ \sum_{i = 1}^n \left\{ v_i^{-1} \cdot X_i + \left( v_i^{-1} \cdot X_i \right)^{-1} \right\} \right]^{-1}.
\end{align*}
For $u = (u_1, \ldots, u_n) \in V$, since $U$ is totally ordered, if $u \not= v$, then there exists $i$ such that $v_i^{-1} \cdot u_i \not= 0$.
Hence, $v_i^{-1} \cdot u_i$ or $(v_i^{-1} \cdot u_i)^{-1}$ is bigger than $0$.
Thus, we have
\begin{align*}
f(u) 
\begin{cases}
= 0 \quad \text{if } u = v,\\
< 0 \quad \text{if } u \not= v.
\end{cases}
\end{align*}
For any subgroup $H$ of $G$, let $f_H := \sum_{h \in H}h(f)$.
By definition, $f_H$ takes zero at and only at elements of $Hv$, and values less than zero at any other elements of $V$.
Therefore, the stabilizer subgroup of $G$ with respect to $f_H$ is $H$, which completes the proof by Proposition \ref{prop:sufficient1}.
\end{proof}

\end{document}